\documentclass{amsart}
\usepackage[utf8]{inputenc}
\usepackage{amsmath, amssymb}
\usepackage{amsthm}
\usepackage[all]{xy}
\usepackage{color}
\definecolor{darkblue}{rgb}{0,0,0.6}
\usepackage[ocgcolorlinks,colorlinks=true, citecolor=darkblue, filecolor=darkblue,
linkcolor=darkblue, urlcolor=darkblue]{hyperref}
\usepackage[nameinlink,capitalise,noabbrev]{cleveref}
\usepackage{bbm}
\makeatletter
\@namedef{subjclassname@2010}{%
  \textup{2010} Mathematics Subject Classification}
\makeatother

\title{The asymptotic dimension of quotients by finite groups}
\author{Daniel Kasprowski}
\address{Max-Planck-Institut für Mathematik, Vivatsgasse 7, 53111 Bonn, Germany}
\email{kasprowski@mpim-bonn.mpg.de}
\thanks{This work was supported by the Max Planck Society.}
\date{\today}

\newcommand{\bbR}{\mathbbm{R}}
\newcommand{\bbC}{\mathbbm{C}}

\newcommand{\bbN}{\mathbbm{N}}

\newcommand{\bbK}{\mathbb{K}}

\newcommand{\mcC}{\mathcal{C}}

\DeclareMathOperator{\diam}{diam}
\DeclareMathOperator{\asdim}{asdim}

\newcommand{\mcU}{\mathcal{U}}

\newcommand{\Fin}{{\mathcal{F}in}}

\numberwithin{equation}{section}
\newtheorem{thm}[equation]{Theorem}
\newtheorem{theorem}[equation]{Theorem}
\newtheorem{prop}[equation]{Proposition}
\newtheorem{cor}[equation]{Corollary}
\newtheorem{lemma}[equation]{Lemma}
\theoremstyle{definition}

\newtheorem{question}[equation]{Question}
\newtheorem{defi}[equation]{Definition}

\newtheorem{rem}[equation]{Remark}

 \newtheoremstyle{TheoremNum}
        {}{}              
        {\itshape}                      
        {}                              
        {\bfseries}                     
        {.}                             
        { }                             
        {\thmname{#1}\thmnote{ \bfseries #3}}
\theoremstyle{TheoremNum}

\begin{document}
\begin{abstract}
Let $X$ be a proper metric space and let $F$ be a finite group acting on $X$ by isometries. We show that the asymptotic dimension of $F\backslash X$ is the same as the asymptotic dimension of $X$.
\end{abstract}
\maketitle
\section{Introduction}
If a metric space $X$ has asymptotic dimension $n$ and $F$ is a finite group acting isometrically on $X$, then it is easy to show that $F\backslash X$ has asymptotic dimension at most $|F|(n+1)-1$. We will show that the asymptotic dimension of the quotient is equal to the asymptotic dimension of $X$.

\begin{theorem}
\label{thm:main}
Let $X$ be a proper metric space and let $F$ be a finite group acting isometrically on $X$. Then $F\backslash X$ has the same asymptotic dimension as $X$.
\end{theorem}

As a corollary we also obtain the following family version of the theorem.

\begin{cor}
\label{cor:main}
Let $\{X_i\}_{i\in I}$ be a set of proper metric spaces and let $\{F_i\}_{i\in I}$ be a set of finite groups such that $F_i$ acts isometrically on $X_i$. If $\{X_i\}_{i\in I}$ has asymptotic dimension $n$ uniformly and there exists $N\in \bbN$ with $|F_i|\leq N$ for all $i\in I$, then $\{F_i\backslash X_i\}$ has asymptotic dimension $n$ uniformly.
\end{cor}

This article was motivated by the following.

In \cite[Theorem A.1]{KasLin} the author proved that for a group $G$ with a finite dimensional classifying space for the family of finite subgroups $\underbar EG$ the $K$-theoretic assembly map
\[H^G_n(\underbar EG;\bbK_R)\to K_n(R[G])\]
is split injective for every ring $R$ if for any proper left-invariant metric on $G$ the family $\{F\backslash G\mid F\leq G~\text{finite}\}$ has finite asymptotic dimension uniformly (or more generally finite decomposition complexity). 

While it is easy to show that if $G$ has finite asymptotic dimension, then for every $n\in \bbN$ the family $\Fin(G)_n:=\{F\backslash G\mid F\leq G, |F|\leq n\}$ has finite asymptotic dimension uniformly, the following question is open.
\begin{question}
\label{ques:asdim}
Has $\Fin(G):=\{F\backslash G\mid F\leq G~\text{finite}\}$ finite asymptotic dimension uniformly if $G$ has finite asymptotic dimension?
\end{question}
By \cref{cor:main} we have $\asdim \Fin(G)_n=\asdim G$ and hence it is not possible to obtain a counterexample to \cref{ques:asdim} by finding a group $G$ for which the uniform asymptotic dimension of $\Fin(G)_n$ goes to infinity with increasing $n\in\bbN$. Note that it is important that the metric is left-invariant and the quotient is taken from the left. For every finite subgroup $F\leq G$ the quotient $G/F$ is quasi-isometric to $G$, but this is in general not true for $F\backslash G$.

From \cref{thm:main} we also obtain the following corollary about the existence of equivariant covers.
\begin{cor}
\label{cor:covers}
Let $X$ be a proper metric space with asymptotic dimension at most $n$ and let $F$ be a finite group acting isometrically on $X$. Then for every $R>0$ there exists an $F$-equivariant, bounded cover of $X$ with Lebesgue number at least $R$ and dimension at most $n$.
\end{cor}

The key ingredient of the proof of \cref{thm:main} is to use Dranishnikov's result \cite[Theorem 6.2]{dranishnikov} comparing the asymptotic dimension with the topological dimension of the Higson corona. We recall the definition of the Higson corona and the comparison result in \cref{sec:higson}. 
\section{Asymptotic dimension}
\label{sec:asdim}
Asymptotic dimension is a coarse invariant of metric spaces introduced by Gromov \cite{gromov}. We begin by giving the definition and collect a few well-known facts.

\begin{defi}
Let $r>0$. A metric space $X$ is the \emph{$r$-disjoint union} of subspaces $X_i,~i\in I$ if $X=\bigcup_{i\in I}X_i$ and for all $x\in X_i, y\in X_j$ with $i\neq j$ we have $d(x,y)>r$. In this case we write
\[X=\bigsqcup_{r\text{-disjoint}}\{X_{i}~|~i\in I\}.\]
\end{defi}

\begin{defi}
A metric $X$ has \emph{asymptotic dimension at most $n$} if for each $r>0$ there exist decompositions
\[X=\bigcup_{j=0}^nU_j,\quad U_i=\bigsqcup_{\lambda\in I_j}^{r-disj.}V_j^{\lambda}\]
such that $\sup\{\diam V_j^{\lambda}\mid j\in\{0,..,n\},\lambda\in I_j\}<\infty$. We denote the asymptotic dimension of $X$ by $\asdim X$.

A set $\{X_i\}_{i\in I}$ of metric spaces has asymptotic dimension at most $n$ if for each $r>0$ there exist decompositions
\[X_i=\bigcup_{j=0}^nU_{i,j},\quad U_{i,j}=\bigsqcup_{\lambda\in I_{i,j}}^{r-disj.}V_{i,j}^{\lambda}\]
such that $\sup\{\diam V_j^{\lambda}\mid i\in I,j\in\{0,..,n\},\lambda\in I_{i,j}\}<\infty$.
\end{defi}
\begin{rem}
Often a set $\{X_i\}_{i\in I}$ with the above property is said to have asymptotic dimension uniformly. We used this convention in the introduction but will omit the word uniformly from now on.
\end{rem}

If a finite group $F$ acts isometrically on a metric space $X$, then we will use the following metric on $F\backslash X$:
\[d(Fx,Fx')=\min_{f\in F}d(x,fx')\]

We will also need an equivalent formulation of finite asymptotic dimension. For this recall the following definitions.
\begin{defi}
A cover $\mcU$ of a metric space $X$
\begin{enumerate}
\item is \emph{bounded} if $\sup_{U\in\mcU}\diam U<\infty$.
\item has \emph{dimension at most $n$} if every $x\in X$ is contained in at most $n+1$ elements of $\mcU$.
\item has \emph{Lebesgue number} at least $R$ if for every $x\in X$ there exists $U\in\mcU$ with $B_R(x)\subseteq U$.
\end{enumerate}
\end{defi}
\begin{prop}[{\cite[Theorem 9.9]{roe}}]
\label{prop:alternative}
Let $X$ be a proper metric space. Then $\asdim X\leq n$ if for each $R>0$ there exists a bounded cover $\mcU$ of $X$ such that no more than $n+1$ members of $\mcU$ meet any ball of radius $R$. Equivalently, for each $R>0$ there exists a bounded cover $\mcU$ of $X$ of dimension at most $n$ and Lebesgue number at least $R$.
\end{prop}
\begin{lemma}
\label{lem:finasdim}
Let $X$ be a proper metric space with $\asdim X=n$ and let $F$ be a finite group acting by isometries on $X$, then the quotient $F\backslash X$ has asymptotic dimension at most $(|F|(n+1))-1$.
\end{lemma}
\begin{proof}
In this proof we will use the alternative description of asymptotic dimension from \cref{prop:alternative}. Let $R>0$ be given and choose a bounded cover $\mcU$ of $X$ of dimension at most $n$ and Lebesgue number at least $R$. 

For every $U\in\mcU$ we have $\diam F\backslash FU<\diam U$ and thus $F\backslash \mcU:=\{F\backslash FU\mid U\in\mcU\}$ is a bounded cover of $F\backslash X$. The ball $B_R(x)$ of radius $R$ around $x\in X$ maps onto the ball $B_R(Fx)\subseteq F\backslash X$ and hence the Lebesgue number of $F\backslash\mcU$ is bigger or equal to the Lebesgue number of $\mcU$.

Futhermore, each of the preimages $fx\in X$ of $Fx\in F\backslash$ is contained in at most $n+1$ elements of $\mcU$. Therefore, $Fx$ is contained in at most $|F|(n+1)$ elements of $F\backslash \mcU$ and has dimension at most $|F|(n+1)-1$.
\end{proof}

For the proof of \cref{cor:main} we will need the following results relating the asymptotic dimension of a set of metric spaces to the asymptotic dimension of a single space.

\begin{lemma}[{\cite[Lemma 2.2]{boxspaces}}]
\label{lemma:countable}
Let $\{X_i\}_{i\in I}$ be a set of metric spaces. Then
\[\asdim \{X_i\}_{i\in I}=\sup\{\asdim\{X_j\}_{j\in J}\mid J\subseteq I~\textnormal{countable}\}.\]
\end{lemma}
\begin{defi}
For a set $\{(X_n,d_n)\}_{n\in\bbN}$ of metric spaces with finite subspaces $Y_n\subseteq X_n$ and a sequence $\{f(n)\}_{n\in\bbN}$ of positive numbers with $f(n)\geq \diam Y_n$ let $S(\{X_n\},\{f_n\})$ denote the disjoint union $S(\{X_n\},\{f_n\})=\bigsqcup_{i\in I}X_i$ with the following metric. For $x\in X_n, y\in X_m$ we have
\[d(x,y)=\left\{\begin{matrix}d_n(x,y)&n=m\\d_n(x,Y_n)+d_m(y,Y_m)+\max\{f(n),f(m)\}&\text{else}\end{matrix}\right.\]
This is indeed a metric since $f(n)\geq \diam Y_n$.
\end{defi}

\begin{prop}
\label{prop:space}
Let $\{(X_n,d_n)\}_{n\in\bbN}$ be as set of metric spaces with finite subspaces $Y_n\subseteq X_n$ and $\{f(n)\}_{n\in\bbN}$ a sequence of strictly increasing positive numbers with $f(n)\geq Y_n$. The metric space $S(\{X_n\},\{f(n)\})$ has the same asymptotic dimension as $\{X_n\}_{n\in\bbN}$ and it is proper if and only if each $X_n$ is proper.
\end{prop}
\begin{proof}
Let $\asdim \{X_n\}_{n\in \bbN}=d$.

By \cite[Proposition 9.13]{roe} for every $N\in\bbN$ we have \[\asdim \bigcup_{n=1}^NX_n=\max\{\asdim X_n\mid n\leq N\}\leq \asdim \{X_n\}_{n\in\bbN},\]
where $\bigcup_{n=1}^NX_n$ is considered as a subset of $S(\{X_n\},\{f_n\})$. 
For $r>0$ choose decompositions
\[X_n=\bigcup_{i=0}^dX_{n,i},~X_{n,i}=\bigsqcup_{\lambda\in J_{n,i}}^{r-disj}Y_{n,i}^\lambda\]
with $\sup\{\diam Y_{n,i}^\lambda\mid n\in\bbN, i\in\{0,\ldots,n\}, \lambda\in J_{n,i}\}<\infty$. Also choose $N>R$ and a decomposition
\[\bigcup_{n=1}^NX_n=\bigcup_{i=0}^dU_i,~U_i=\bigsqcup_{\lambda\in J_i}^{r-disj}U_i^\lambda\]
with $\sup\{\diam U_i^\lambda\mid i\in\{0,\ldots,n\}, \lambda\in J_i\}<\infty$. Then we can consider the decomposition
\[S(\{X_n\},\{f(n)\})=\bigcup_{i=0}^d\left(U_i\cup\bigcup_{n=N+1}^\infty X_{n,i}\right)\]
with
\[U_i\cup\bigcup_{n=N+1}^\infty X_{n,i}=\bigsqcup_{\lambda\in J_i}^{r-disj}U_i^\lambda\sqcup^{r-disj}\bigsqcup^{r-disj}_{n>N, \lambda \in J_{n,i}}Y_{n,i}^\lambda,\]
where the $r$-disjointness follows from $f(n)\geq n$ and the definition of $S(\{X_n\},\{f(n)\})$.

This shows that $\asdim S(\{X_n\},\{f_n\})\leq \asdim \{X_n\}_{n\in\bbN}$. 
On the other hand any such decomposition for $S(\{X_n\},\{f(n)\})$ can be restricted to the $X_n$ to prove that $\asdim \{X_n\}_{n\in\bbN}\leq \asdim S(\{X_n\},\{f(n)\})$.

If $S(\{X_n\},\{f(n)\})$ is proper than all the subspaces $X_n$ are proper as well. For $N>r$ the ball $B_r(x)\subseteq S(\{X_n\},\{f(n)\})$ for $x\in X_n$ is contained in $(B_r(x)\cap X_n)\cup\bigcup_{i=1}^{N}B_r(Y_i)$. This is a union of finitely many balls and thus compact if $X_i$ is proper for all $i\leq N$.
\end{proof}
\begin{lemma}
	\label{lem:covers}
Let $X$ be a metric space and let $F$ be a finite group acting isometrically on $X$. If $\asdim F\backslash X=n$, then for every $R>0$ there exists an $F$-equivariant, bounded cover of $X$ with Lebesgue number at least $R$ and dimension at most $n$. In particular, $\asdim X\leq \asdim F\backslash X$.
\end{lemma}
\begin{proof}
	By \cref{prop:alternative} there is a bounded cover $\mcU$ of $F\backslash X$ of dimension at most $n$ and Lebesgue number at least $R$. Let $p\colon X\to F\backslash X$ be the projection. The cover $\{p^{-1}(U)\mid U\in\mcU\}$ is an $F$-equivariant cover of dimension at most $n$ and with Lebesgue number at least $R$ but it might not be bounded.
	
	Let $s:=\sup\{\diam U\mid U\in\mcU\}\geq R$. Then for each $x\in p^{-1}(U)$ we have $p^{-1}(U)\subseteq \bigcup_{f\in F} B_s(fx)$. For $x\in X$ let $F_{x,s}$ be the subgroup of $F$ generated by $\{f\in F\mid d(x,fx)\leq 4s\}$. For $x\in U$ define $U_x:=p^{-1}(U)\cap \bigcup_{f\in F_{x,s}}B_s(fx)$.
	
	We will first show that $\diam U_x<4s(|F|+1)$. Given $y,y'\in U_x$ there are $f,f'\in F_{x,s}$ with $d(y,fx)\leq s$ and $d(y',f'x)\leq s$. Therefore, $\diam U_x\leq \diam F_{x,s}x+2s$. We have $\diam F_{x,s}x=\max_{f\in F_{x,s}}d(x,fx)$ and every $f\in F_{x,s}$ can be written as $f=f_1\ldots f_k$ with $d(x,f_kx)\leq 4s$ and $k\leq |F|$. Hence $d(x,fx)\leq \sum_{i=1}^kd(x,f_ix)\leq 4ks$.
	
	For $U\in\mcU$ choose $x_U\in U$. Then $p^{-1}(U)=\bigcup_{f\in F}U_{fx_U}$. If $f^{-1}f'\in F_{x_U,s}$ then by definition we have $U_{fx_U}=U_{f'x_U}$. Now suppose $y\in B_s(U_{fx_U})\cap B_s(U_{f'x_U})$, then there are $h\in F_{fx_U,s}$ and $h'\in F_{f'x_U,s}$ with $d(y,hfx_U)\leq 2s$ and $d(y,h'f'x_U)\leq 2s$. Thus $d(hfx_U,h'f'x_U)\leq 4s$ and $f^{-1}h^{-1}h'f'\in F_{x_U,s}$. If $h\in F_{fx_U,s}$ then $d(f^{-1}hfx_U,x_U)=d(hfx_U,fx_U)\leq 4s$ and $f^{-1}hf\in F_{x_U,s}$. It follows that
	\[f^{-1}f'=(f^{-1}hf)(f^{-1}h^{-1}h'f')((f')^{-1}h'^{-1}f')\in F_{x_U,s}\]
	and thus $p^{-1}(U)$ is the $2s$ disjoint union of the $U_{fx_U}$ where one $f\in F$ per coset $F/F_{x_U,s}$ is chosen. 
	
	The cover
	\[\{U_{fx_U}\mid U\in\mcU, f\in F\}\]
	now is bounded, of dimension at most $n$ and has Lebesgue number at least $R$.
\end{proof}
\section{The Higson corona}
\label{sec:higson}
The following definitions can for example be found in \cite{dranishnikov} and \cite{roe}.

Let $X$ be a proper metric space. Given a bounded and continuous function $f\colon X\to \bbC$ and $R>0$, we define $\text{Var}_Rf\colon X\to\bbR$ by \[\text{Var}_Rf(x):=\sup\{|f(x)-f(y)|\mid d(x,y)<R\}.\] A function $f$ is a \emph{Higson function} if $\text{Var}_Rf\in\mcC_0(X)$ for every $R>0$. Furthermore, $\mcC_h(X)$ is the $C^*$-algebra of all Higson functions. By the Gelfand-Naimark theorem $\mcC_h(X)\supseteq \mcC_0(X)$ is isomorphic to the $C^*$-algebra of continuous functions $\mcC_0(hX)=\mcC(hX)$ for some compactification $hX$ of $X$. The compact space $hX$ is unique up to homeomorphism and is called the \emph{Higson compactification} of $X$. This yields a functor from the category of proper metric spaces and coarse maps to the category of compact Hausdorff spaces. The \emph{Higson corona} is defined as $\nu X:=hX\setminus X$. If a group $F$ acts on $X$ we denote the $C^*$-algebra of $F$-invariant Higson functions by $\mcC_h(X)^F$.
%

\begin{prop}
\label{prop:action}
Let $F$ be a finite group acting isometrically on a proper metric space $X$. By functoriality this induces an action of $F$ on $hX$. We have the following isomorphisms
\[\mcC(F\backslash hX)\cong\mcC(hX)^F\cong\mcC_h(X)^F\cong \mcC_h(F\backslash X)\cong \mcC(h(F\backslash X)).\]
Since $F$ is finite, the quotient $F\backslash hX$ is Hausdorff and thus $F\backslash\nu X$ and $\nu(F\backslash X)$ are homeomorphic.
\end{prop}
\begin{proof}
All but the third isomorphism follow directly from the definition of $hX$ and the fact that a map $X\to\bbC$ is $F$-equivariant if and only if it factors through $F\backslash X\to \bbC$.

Since $F$ is finite, the pre-image of every compact subset $K\subseteq F\backslash X$ is again compact. Therefore, if $f\colon F\backslash X\to\bbC$ is a Higson function, then so is $X\to F\backslash X\xrightarrow{f}\bbC$. This implies the third isomorphism.
\end{proof}

The \emph{topological dimension} $\dim X$ of a topological space $X$ is the smallest $n\in\bbN$ such that every open cover of $X$ has an open refinement of dimension at most $n$.

For the proof of the main theorem we need the following comparison.
\begin{thm}[{\cite[Theorem 6.2]{dranishnikov}}]
\label{thm:compare}
If a proper metric space $X$ has finite asymptotic dimension, then $\dim(\nu X)=\asdim X$.
\end{thm}
\section{Proof of the main theorem}
\label{sec:mainproof}
The last ingredient we need is the following proposition.
\begin{prop}[{\cite[Proposition 9.2.16]{pears}}]
\label{prop:fintoone}
Let $X,Y$ be weakly paracompact, normal spaces. Let $f\colon X\to Y$ be a continuous, open surjection. If for every point $y\in Y$ the pre-image $f^{-1}(y)$ is finite, then
$\dim(X) = \dim(Y).$
\end{prop}
\begin{proof}[Proof of \cref{thm:main}]
By \cref{lem:covers} $\asdim X=\infty$ implies $\asdim F\backslash X=\infty$.

Hence let $X$ be a proper metric space with $\asdim X=n<\infty$ and let $F$ be a finite group acting isometrically on $X$. By \cref{lem:finasdim} the quotient $F\backslash X$ has again finite asymptotic dimension and thus by \cref{thm:compare} its asymptotic dimension is the same as the dimension of $\nu(F\backslash X)$. By \cref{prop:action} we have a homeomorphism
\[F\backslash\nu(X)\cong \nu(F\backslash X).\]
The map $\nu X\to F\backslash \nu X$ is surjective and open, since it is the projection under a group action. The space $\nu X$ is a compact Hausdorff space and hence also $F\backslash \nu(X)$ is compact and Hausdorff. In particular both spaces are paracompact and normal. Now using \cref{prop:fintoone} and \cref{thm:compare} together with the above we get
\[\asdim F\backslash X=\dim \nu (F\backslash X)=\dim F\backslash \nu X=\dim \nu X=\asdim X=n.\qedhere\]
\end{proof}
\begin{proof}[Proof of \cref{cor:main}]
By \cref{lemma:countable} it suffices to consider the case where $I$ is countable and thus we will use $\bbN$ as index set instead.
Since $|F_n|\leq N$ there exist finitely many finite groups $H_j,~j=1,\ldots,k$ such that each $F_n$ is isomorphic to some $H_j$. Now let $H:=\bigoplus_{j=1}^kH_j$ act on $X_n$ as follows. Choose one $H_j$ isomorphic to $F_n$ and let $H_j$ act on $X_n$ using this isomorphism. All other components act trivially.
Then $\{H\backslash X_n\}_{n\in \bbN}$ is isometric to $\{F_n\backslash X_n\}_{n\in \bbN}$. Choose any sequence of points $x_n\in X_n$ and let $Y_n:=Hx_n$ and $f(n):=\diam Y_n+n$. 

Then $H$ acts componentwise on $S(\{X_n\},\{f(n)\})$ and this action is isometric since $d(fx,Y_n)=d(x,Y_n)$ for all $x\in X_n, f\in H$. Furthermore, $H\backslash S(\{X_n\},\{f(n)\})$ is isometric to $S(\{H\backslash X_n\},\{f(n)\})$ where for the definition of the later we consider the subspace $\{Hx_n\}\subseteq H\backslash X_n$.
By \cref{thm:main} and \cref{prop:space} we have
\begin{align*}
\asdim\{X_n\}_{n\in \bbN}&=\asdim S(\{X_n\},\{f(n)\})=\asdim H\backslash S(\{X_n\},\{f(n)\})\\
&=\asdim S(\{H\backslash X_n\},\{f(n)\})=\asdim \{H\backslash X_n\}_{n\in \bbN}.\qedhere\end{align*}
\end{proof}
\cref{cor:covers} directly follows from \cref{thm:main} and \cref{lem:covers}.
\bibliographystyle{amsalpha}
\bibliography{fqFDC}

\providecommand{\bysame}{\leavevmode\hbox to3em{\hrulefill}\thinspace}
\providecommand{\MR}{\relax\ifhmode\unskip\space\fi MR }
\providecommand{\MRhref}[2]{%
  \href{http://www.ams.org/mathscinet-getitem?mr=#1}{#2}
}
\providecommand{\href}[2]{#2}
\begin{thebibliography}{{Dra}00}

\bibitem[{Dra}00]{dranishnikov}
A.N. {Dranishnikov}, \emph{{Asymptotic topology.}}, {Russ. Math. Surv.}
  \textbf{55} (2000), no.~6, 1085--1129 (2000); translation from us (English).

\bibitem[FSW]{boxspaces}
Martin Finn-Sell and Jianchao Wu, \emph{The asymptotic dimension of box spaces
  for elementary amenable groups}, arXiv:1508.05018v1.

\bibitem[Gro93]{gromov}
M.~Gromov, \emph{Asymptotic invariants of infinite groups}, Geometric group
  theory, {V}ol.\ 2 ({S}ussex, 1991), London Math. Soc. Lecture Note Ser., vol.
  182, Cambridge Univ. Press, Cambridge, 1993, pp.~1--295. \MR{1253544
  (95m:20041)}

\bibitem[Kas]{KasLin}
Daniel Kasprowski, \emph{On the {K}-theory of linear groups},
  arXiv:1502.02240v2, to appear in Annals of K-theory.

\bibitem[{Pea}75]{pears}
A.R. {Pears}, \emph{{Dimension theory of general spaces.}}, {Cambridge etc.:
  Cambridge University Press. XII, 428 p. (1975).}, 1975.

\bibitem[Roe03]{roe}
John Roe, \emph{Lectures on coarse geometry}, University Lecture Series,
  vol.~31, American Mathematical Society, Providence, RI, 2003. \MR{2007488
  (2004g:53050)}

\end{thebibliography}
\end{document}